\documentclass[11pt]{amsart}

\usepackage{pythagorasheader}
\usepackage[foot]{amsaddr} 

\newcommand{\p}{\mathfrak{p}} 
\newcommand{\zetazeta}{\zeta_7+\zeta_7^{-1}}
\newcommand{\Kivix}{K^{(49)}}
\newcommand{\Odya}{\O_{(2)}}
\newcommand{\Olocal}{\mathfrak{O}}

\begin{document}

\author{Jakub Kr\'asensk\'y}

\title{A cubic ring of integers with the smallest Pythagoras number}

\address{Charles University, Faculty of Mathematics and Physics, Department of Algebra,\newline
Sokolovsk\'{a}~83, 18600 Praha 8, Czech Republic}
\email{\hangindent=8em\hangafter=1 krasensky@seznam.cz}


\thanks{The author aknowledges partial support by project PRIMUS/20/SCI/002 from Charles University, by Czech Science Foundation GA\v{C}R, grant 21-00420M, by projects UNCE/SCI/022 and GA UK No.\ 742120 from Charles University, and by SVV-2020-260589.}

\begin{abstract}
We prove that the ring of integers in the totally real cubic subfield $\Kivix$ of the cyclotomic field $\Q(\zeta_7)$ has Pythagoras number equal to $4$. This is the smallest possible value for a totally real number field of odd degree. Moreover, we determine which numbers are sums of integral squares in this field, and use this knowledge to construct a diagonal universal quadratic form in five variables. 
\end{abstract}

\setcounter{tocdepth}{2}  \maketitle 

\section{Introduction}

In this note, we shall study sums of integral squares in the cubic field $\Kivix = \Q(\zetazeta)$, which can be characterised e.g.\ as the unique number field of discriminant $49$ or as the maximal totally real subfield of the seventh cyclotomic field. In the main Theorem \ref{th:main}, we show that any sum of integral squares can in fact be written as a sum of four such squares -- in the terminology introduced below, the Pythagoras number of $\O_{\Kivix}$ is $4$. In Theorem~\ref{th:odd>3} we recall the generally known fact that this is the smallest possible value among totally real fields of odd degree. Moreover, with the information gained during the proof of the main theorem, it is fairly simple to fully characterise the numbers which are sums of integral squares in $\Kivix$ (Theorem \ref{th:characterisation}) and to use this to construct a diagonal universal quadratic form in five variables (Corollary \ref{co:universal}).

Let us start by summarising the known results. For standard definitions and notation, in particular regarding number fields, local fields and quadratic forms, we refer the reader to \cite{Mi} and \cite{OMeara}.

\smallskip

When $R$ is a commutative ring, $\sum R^2$ denotes the set of all elements of $R$ which can be written as a sum of squares, and for $\alpha \in \sum R^2$, its \emph{length} $\ell(\alpha)$ is the smallest integer $n$ such that $\alpha$ can be written as a sum of $n$ squares. The \emph{Pythagoras number} of $R$ is then defined as the largest length occurring in $\sum R^2$, i.e.
\[
\P(R) = \sup \bigl\{\ell(\alpha) : \alpha \in \textstyle{\sum} R^2\bigr\}.
\]
In this notation, Lagrange's four square theorem can be written as $\P(\Z)=4$. The Pythagoras number is primarily studied for fields \cite{Ho, TVY, BV, BL, BGV, Hu}, as this is arguably the easier case, but there are several results about the Pythagoras number of rings of integers $\O_K$ of a number field $K$ as well. In fact, if $K$ is not totally real, then $\P(\O_K) \leq 4$ \cite{Pf}. For totally real fields, the situation is dramatically different: By Scharlau \cite{Sch}, there are number fields $K$ with arbitrarily large $\P(\O_K)$. On the other hand, $\P(\O_K)$ is bounded by a constant depending only on the degree $d=[K : \Q]$, as shown by Kala and Yatsyna \cite{KY}. This constant is the so-called $g$-invariant $g_{\Z}(d)$, see e.g.\ \cite{KO}; thus for $1 \leq d \leq 5$, this bound is $d+3$, and for $d=6$, its value is $10$.

For a given number field $K$, it is fairly simple to obtain lower bounds for $\P(\O_K)$, since finding the length of any given $\alpha$ is only a computational task; the main content of several recent papers \cite{Ti, KRS, Kr} is finding good lower bounds for infinite families of fields. Obtaining tight upper bounds is much more difficult. The result of Kala and Yatsyna was generalised in \cite{KRS} to include $g$-invariants of any subfield of $K$, which in particular yielded $\P(\O_K) \leq 5$ for quadratic extensions of $\Q(\sqrt5)$; however, very few values of $g$-invariants of totally real orders are known, so this generalisation is of limited use.

Since cubic fields have no subfields besides $\Q$, the only easily available upper bound for their Pythagoras number is $\P(\O_K) \leq g_{\Z}(3) = 6$. Recently, Tinková \cite{Ti} showed that this upper bound is attained in infinitely many totally real cubic fields $K$. More precisely, she studied the family of so-called \emph{simplest cubic fields} $\Q(\rho_a)$, introduced by Shanks \cite{Sh}, where $\rho_a$ is a root of the polynomial $x^3-ax^2-(a+3)x-1$, $a \geq -1$, and proved that $\P(\Z[\rho_a]) = 6$ for $a \geq 3$. (This order is the full ring of integers for a positive proportion of $a$.)

For $a=0,1,2$, Tinková exhibits elements of length $5$, and for $a=-1$ she only gives $\ell(7)=4$, noting that using a computer program, she checked that no element with a reasonably small trace has bigger length. Therefore, it is natural to conjecture that $\P(\Z[\rho_{-1}])=4$. Our Theorem \ref{th:main} confirms this hypothesis, since $\Q(\rho_{-1}) = \Kivix$. Among totally real cubic fields, our result provides the first one where $\P(\O_K) \leq 4$, and by Tinková's result we know that it is the only simplest cubic field where the order $\Z[\rho_a]$ has Pythagoras number less than $5$.

In total, very few totally real fields $K$ with $\P(\O_K) \leq 4$ are known: For $K = \Q(\sqrt2), \Q(\sqrt3)$ and $\Q(\sqrt5)$ one has $\P(\O_K)=3$, while for $K = \Q, \Q(\sqrt6)$ and $\Q(\sqrt7)$ it is known that $\P(\O_K)=4$. For all other real quadratic fields, the maximal possible value $5 = g_{\Z}(2)$ is attained. The paper \cite{KRS} proves that there are at most seven real biquadratic fields with $\P(\O_K) \leq 4$, conjecturing that these seven fields indeed satisfy the inequality. This weak evidence may lead one to conjecture that in fact, there are only finitely many totally real $K$ with $\P(\O_K) \leq 4$.

We conclude this section by combining two well-known facts in order to show $\P(\O_K) \geq 4$ for every totally real number field of odd degree. The result is stated for general orders:

\begin{theorem}\label{th:odd>3}
Let $K$ be a totally real number field with $[K : \Q]$ odd. Then $\P(K) = 4$. If~$\O$ is an order in $K$, then $\P(\O) \geq 4$.
\end{theorem}
\begin{proof}
The fact that $\P(K)=4$ can be found in \cite[Ex.\ ~XI.5.9(2)]{La}. It is an application of the local-global principle, and the core of the argument is as follows: Since there is an odd number of real embeddings, by Hilbert's reciprocity law there exists at least one finite place where the form $x^2+y^2+z^2$ is anisotropic.

It is also well known that if $R$ is any integral domain and $F$ its field of fractions, then $\P(R) \geq \P(F)$. Indeed, if $\frac{\alpha}{\beta}$ is not a sum of $n$ squares in $F$, then the same holds for $\alpha\beta$ in $F$ and thus also in $R$. Now the proof is concluded since the field of fractions of $\O$ is $K$.
\end{proof}

\section{Preliminaries and results}

We have already defined the Pythagoras number, the set $\sum R^2$ and the length $\ell(\,\cdot\,)$. By $I_n$ we mean the quadratic form $x_1^2 + \cdots + x_n^2$. For brevity, we sometimes denote squares simply by $\square$. The symbol $h(\varphi)$ stands for the class number of the quadratic form $\varphi$. For its definition, as well as for other standard terminology and notation, we refer the reader to O'Meara's book \cite{OMeara}. The starting point of this paper is the following result, proven by analytical methods in \cite{Pe2}:

\begin{theorem}[Peters, 1977] \label{th:peters}
There are only six totally real number fields where $h(I_3)=1$; $\Kivix$ is one of them. (The other fields in question are $\Q$, $\Q(\sqrt2)$, $\Q(\sqrt5)$, $\Q(\sqrt{17})$ and the cubic field $K^{(148)}$ with discriminant $148$.)
\end{theorem}

From now on, we will leave general number fields and focus only on the field with discriminant $49$. Let $\zeta_7 = \mathrm{exp}(\frac{2\pi\mathrm{i}}{7})$ be the primitive seventh root of unity and let $\Kivix$ stand for the totally real cubic field $\Q(\zetazeta)$. Usually we denote its ring of integers $\Z[\zetazeta]$ simply by $\O$. Later, we will also use the element $\rho = \zetazeta$ which generates $\O$; its minimal polynomial is $x^3+x^2-2x-1$. Our main theorem is the following:

\begin{theorem} \label{th:main}
$\P\bigl(\Z[\zetazeta]\bigr)=4$.
\end{theorem}

There are many remarkable properties of $\Kivix$; for example, it has class number $1$, it is a Galois extension of $\Q$, and every other cubic Galois extension of $\Q$ has larger discriminant (in absolute value).

Thanks to Theorem \ref{th:peters}, sums of three integral squares in $\O$ satisfy the local-global principle; hence, our main tools are the completions $\O_\p$ of the ring of integers $\O$ at places $\p$. In particular, we will examine the only dyadic place of $\O$, corresponding to the prime ideal $(2)$. (The fact that $2$ is indeed a prime of $\O = \Z[\zetazeta]$ is easily checked using \cite[Th.\ ~3.41]{Mi}.)

\smallskip

In general, local conditions do not suffice to describe the set $\sum \O_K^2$; usually there are elements which are not a sum of squares (\enquote{globally}) despite being a sum of integral squares at all completions. Scharlau calls them \enquote{Ausnahmeelemente} (i.e.\ \emph{exceptional elements}) and examines them in depth in \cite{Sch2} -- for example, he shows that up to multiplication by squares of units, there are always only finitely many of them; on the other hand, he proves that in every cubic order there is at least one such element. As the discriminant of $K$ is odd, the only local condition for a sum of squares is total positivity \cite[Kap.\ ~0]{Sch2}. In our second result, we characterise the set of sums of squares, showing that up to multiplication by units there is only one exceptional element, namely $1+\rho+\rho^2$. (Remember that $\rho=\zetazeta$.)

\makeatletter
\newcommand{\mylabel}[2]{#2\def\@currentlabel{#2}\label{#1}}
\makeatother 

\begin{theorem} \label{th:characterisation}
Let $\alpha \in \O$ be totally positive. Then the following statements are equivalent:
 \begin{enumerate}
     \item[\mylabel{a}{(1a)}] $\alpha \notin \sum \O^2$.
     \item[\mylabel{b}{(1b)}] $\alpha$ is not a sum of four integral squares.
     \item[\mylabel{c}{(2a)}] The norm of $\alpha$ is $7$.
     \item[\mylabel{d}{(2b)}] $\alpha = u^2(1+\rho+\rho^2)$ for a unit $u \in \O$.
     \item[\mylabel{e}{(2c)}] $\alpha$ is an indecomposable element and not a square.
 \end{enumerate}
\end{theorem}

The proof of this result, as well as the therein contained notion of \emph{indecomposable elements}, is to be found in Section \ref{se:characterisation}. Note that while a full characterisation of $\sum \O_K^2$ is not usually known, for real quadratic fields, the problem is solved in \cite[Satz 2]{Pe}.

\smallskip

We can immediately state a simple corollary of our characterisation of $\sum \O^2$. Remember that a totally positive definite quadratic form is called \emph{universal} if it represents all totally positive integers. In \cite{KY}, it is proven that $\Kivix$ has the very rare property of admitting a universal quadratic form with rational integral coefficients (the form in question has four variables: $x^2+y^2+z^2+w^2+xw+yw+zw$). With our knowledge, one easily constructs a universal quadratic form in five variables:

\begin{corollary} \label{co:universal}
The totally positive definite quadratic form $x_1^2 + x_2^2 + x_3^2 + x_4^2 + (1+\rho+\rho^2)x_5^2$ is universal over $\O$.
\end{corollary}

While the form given in \cite{KY} has less variables, their form is \emph{non-classical}, i.e.\ the corresponding symmetric matrix contains non-integral entries. Our form is not only classical, but even diagonal. For monogenic simplest cubic fields, \cite[Th.\ 1.1]{KT} provides a construction of a diagonal universal quadratic form, which for $\O_{\Kivix}$ requires $12$ variables, so our form is much simpler. To the best of our knowledge, no classical universal form in five or less variables in a totally real field of odd degree has been previously known. (By a well-known argument involving Hilbert reciprocity, at least four variables are necessary for fields of odd degree \cite{EK}.)

\section{Sums of three squares}

It is well known \cite[102:5]{OMeara} that every quadratic form with class number $1$ satisfies the local-global principle. This means that a number is represented over $\O$ if and only if it is represented over $\O_{\p}$ for all places $\p$. While the sum of four squares $I_4$ does not have this property (we will later prove that $1+\rho+\rho^2$ is not a sum of any number of squares, while being a sum of four squares everywhere locally), we shall exploit that $I_3$ satisfies the local-global principle to fully determine which numbers are a sum of three integral squares in $\Kivix$:

\begin{proposition} \label{pr:localglobal}
In $\O = \Z[\zetazeta]$, the quadratic form $I_3$ represents a nonzero number $\alpha$ if and only if $\alpha$ satisfies both the following conditions:
 \begin{enumerate}
     \item $\alpha$ is totally positive;
     \item at the dyadic place $\Odya$, $\alpha \neq -t^2$ for all $t \in \Odya$.
 \end{enumerate}
\end{proposition}

The proof is given at the end of this section. To prepare for it, we need to examine the dyadic completion $\Odya$. We state our results more generally for the ring of integers $\Olocal$ in a dyadic local field $L$ with a few additional properties, since this generality makes the proofs clearer. We start with a useful observation:

\begin{observation}
Let $\Olocal$ be any commutative ring where $2$ is a prime element. Then the following statements are equivalent for $x,y \in \Olocal$:
 \begin{enumerate}
     \item $x \equiv y \pmod2$; \label{it:1}
     \item $x^2 \equiv y^2 \pmod4$; \label{it:2}
     \item $x^2 \equiv y^2 \pmod2$. \label{it:3}
 \end{enumerate}
Indeed, the implications (\ref{it:1}) $\Rightarrow$ (\ref{it:2}) $\Rightarrow$ (\ref{it:3}) are trivial. For (\ref{it:3}) $\Rightarrow$ (\ref{it:1}), one rewrites the condition (\ref{it:3}) as $2 \mid (x-y)^2$ and uses the fact that $2$ is a prime.
\end{observation}

This observation can further be exploited to get the following lemma:

\begin{lemma} \label{le:sumoftwosquares}
Let $\Olocal$ be the ring of integers of a dyadic local field $L$ where $2$ is a prime element and the degree $[L : \Q_2]$ is odd. Let $u,v \in \Olocal$ be units. Then:
 \begin{enumerate}
     \item $u^2 + v^2 \neq \square$;
     \item there are no $y,z \in \O$ such that $u^2+y^2+z^2 \equiv 0 \pmod4$.
 \end{enumerate}
\end{lemma}
\begin{proof}
Assume that, contrary to the first statement, $u^2+v^2=w^2$ for some $w \in \Olocal$. Then $w^2 \equiv (u+v)^2 \pmod2$, so the previous observation yields $w^2 \equiv (u+v)^2 \pmod4$.

Plugging this into $u^2+v^2 = w^2$ gives $2uv \equiv 0 \pmod4$. This is equivalent to $uv \equiv 0 \pmod2$, which cannot happen since $u,v$ are units. Thus the first part is proven.

Similarly, if the second statement is violated, we write $z^2 \equiv (u+y)^2 \pmod2$, so the above observation yields $z^2 \equiv (u+y)^2 \pmod4$. Plugging in, we get $2u^2 + 2y^2 + 2uy \equiv 0 \pmod4$. This is equivalent to $u^2 + uy + y^2 \equiv 0 \pmod2$; as $u$ is a unit, it follows that there exists a root of $t^2+t+1$ modulo $2$, i.e.\ in $\Olocal/(2)$.

Since $2$ is the prime element, $\Olocal/(2)$ is the residue field. It is isomorphic to $\mathbb{F}_{2^d}$, where $d=[L:\Q_2]$ is by assumption an odd number. However, this is a contradiction -- no root of an irreducible quadratic polynomial over $\mathbb{F}_2$ can be contained in an extension of odd degree.
\end{proof}

We were building towards the following characterisation of sums of three squares in~$\Odya$ and more generally in the ring of integers in any unramified dyadic local field of odd degree:

\begin{lemma} \label{le:dyadicI3}
Let $L$ be any dyadic local field where $2$ is a prime element and the degree $[L : \Q_2]$ is odd. Then, over its ring of integers $\Olocal$, the form $I_3$ represents exactly those elements which are not equal to $-t^2$ for any $t \in \Olocal$.
\end{lemma}
\begin{proof}
Bear in mind that a square is integral if and only if the squared number is integral. Now, the proof has two parts. First we show that if a number in $\Olocal$ is a sum of three squares in $L$, then these squares are necessarily integral. Afterwards we use the known results about representations over a local field.

Let us examine whether a sum of three non-integral squares can belong to $\Olocal$. We use the fact that $2$ is the prime element, and multiply the equality by a suitable power of $2$, to reformulate the question as follows: Is it possible to find $x,y,z \in \Olocal$, with $x$ being a unit, so that $x^2+y^2+z^2 \equiv 0 \pmod4$? By Lemma \ref{le:sumoftwosquares} (2), the answer is no.

We have just shown that if the sum of three squares is integral, then the squared numbers are integral as well (which, as a corollary, also yields that $0$ is not non-trivially represented, i.e.\ $I_3$ is anisotropic). Therefore, we now only have to examine the numbers represented by $I_3$ over the field $L$ instead of over the ring $\Olocal$. Since $I_3$ is anisotropic and its determinant is a square, \cite[63:21]{OMeara} yields that the only not-represented elements are minus squares.
\end{proof}

With this dyadic preparation, we are ready to characterise sums of three squares in $\O$.

\begin{proof}[Proof of Proposition \ref{pr:localglobal}]
By Theorem \ref{th:peters}, the form $I_3$ has class number $1$, so it satisfies the local-global principle. Thus, it suffices to check which numbers are represented locally.

In the real embeddings, every positive number is already a square, while negative numbers cannot be written as any number of squares. In a non-dyadic completion, every ternary unimodular form is universal \cite[92:1b]{OMeara}. Thus it remains to examine the only dyadic place $\Odya$. This was done in Lemma \ref{le:dyadicI3} -- this lemma applies since $(2)$ is inert, so $[\Kivix_{(2)} : \Q_2]$ is equal to $[\Kivix : \Q]$, which is odd.
\end{proof}

\section{The main proof}

Now we fully understand sums of three integral squares. To complete the proof that every sum of squares in $\Z[\zetazeta]$ is already a sum of four squares, one has to examine two types of problematic elements: Those which are equal to minus square of a unit in $\Odya$, and those which are minus square of a number divisible by $2$ in $\Odya$. To deal with the latter type, one simple trick suffices:

\begin{observation} \label{ob:2multiples}
Let $\alpha \in \O$ be totally positive and let it be equal to $-(2t)^2$ at the dyadic place. Then $\ell(\alpha)=4$, i.e.\ it is a sum of four but not of three squares.

Indeed, the number $\alpha/2$ is totally positive and it is clearly not equal to $-\square$ at the dyadic place, hence by Proposition \ref{pr:localglobal} we have $\alpha/2 = x^2 + y^2 + z^2$ in $\O$. Then $\alpha = (x+y)^2 + (x-y)^2 + z^2 + z^2$. On the other hand, by the same proposition, $\alpha$ is not a sum of three squares.
\end{observation}

As a side note, one could also use the universality of the quadratic form $x^2 + y^2 + z^2 + w^2 + xw + yw + zw$ to deduce that every totally positive number in $2\O$ is a sum of four squares, since $2(x^2+y^2+z^2+w^2+xw+yw+zw) = (x+y+w)^2+(x-y)^2+(z+w)^2+z^2$. Either way, combining Proposition \ref{pr:localglobal} with the above observation, we obtain the following neat statement, which is the basis of the proofs of Theorems \ref{th:main} and \ref{th:characterisation}:

\begin{corollary} \label{co:minusunitbad}
If a totally positive number in $\O$ is not a sum of four integral squares, then in $\Odya$ it is equal to $-u^2$ where $u$ is a unit.
\end{corollary}

Finally, we are ready to prove that the Pythagoras number of $\O$ is $4$:

\begin{proof}[Proof of Theorem \ref{th:main}]
Let $\alpha \in \sum \O^2$. Clearly, $\alpha$ is totally non-negative, so in most cases Corollary \ref{co:minusunitbad} applies and $\alpha$ is a sum of four squares. It remains to deal with the case when at the dyadic place, $\alpha = -u^2$ with $u$ a unit. In this case, the condition that $\alpha$ is totally positive is not enough for its being a sum of squares (compare Theorem \ref{th:characterisation}). However, we know that $\alpha = \sum x_i^2$ in $\O$ for some $x_i$. We shall show that if $x_j$ is a dyadic unit, then $\alpha - x_j^2$ is a sum of three squares in $\O$, implying that $\alpha$ is a sum of four squares. This will conclude the proof, since at least one of the summands must be a unit.

For the sake of contradiction, suppose that $\sum_{i \neq j} x_i^2$ is not a sum of three squares in $\O$. It is totally positive, so by Proposition \ref{pr:localglobal} it must be $-\square$ in $\Odya$. However, the equality $-u^2 = x_j^2 - \square$ in $\Odya$ is impossible by Lemma \ref{le:sumoftwosquares} (1) since $u$ a $x_j$ are units.
\end{proof}

\section{Characterisation of sums of squares} \label{se:characterisation}

In this section we are going to determine which numbers can be written as a sum of squares (by Theorem \ref{th:main} they can in fact be represented as a sum of \emph{four} squares). To achieve this, we shall need the following notion: A totally positive number in $\O$ is called \emph{indecomposable} if it cannot be written as a sum of two totally positive elements of $\O$.

The indecomposable elements are quite difficult to study; it is a significant success that \cite[Th.\ 1.2]{KT} fully characterises them for the order $\Z[\rho_a]$ in the simplest cubic fields. For our case, their result reads as follows (recall that $\rho = \zetazeta$):

\begin{lemma}[Kala, Tinková]
Up to multiplication by squares of units, the only indecomposables of $\O$ are $1$ and $1+\rho+\rho^2$.
\end{lemma}

We use this lemma to provide a very explicit description of the set $\sum \O^2$, namely: a totally positive $\alpha$ is an exceptional element if and only if it satisfies the simple condition \ref{c} or equivalently \ref{d}. However, note that the core of the proof does not use the indecomposable elements explicitly -- without the lemma, we would still have the equivalence of \ref{a} (and \ref{b}) with \ref{e}.

\setcounter{section}{2}
\setcounter{theorem}{2}

\begin{theorem}
Let $\alpha \in \O$ be totally positive. Then the following statements are equivalent:
 \begin{enumerate}
     \item[(1a)] $\alpha \notin \sum \O^2$.
     \item[(1b)] $\alpha$ is not a sum of four integral squares.
     \item[(2a)] The norm of $\alpha$ is $7$.
     \item[(2b)] $\alpha = u^2(1+\rho+\rho^2)$ for a unit $u \in \O$.
     \item[(2c)] $\alpha$ is an indecomposable element and not a square.
 \end{enumerate}
\end{theorem}
\begin{proof}
The new input in this proof will be the equivalence between \ref{a} and \ref{e}. Most of the other implications are clear by now:

By Theorem \ref{th:main}, \ref{a} and \ref{b} are equivalent. Statements \ref{d} and \ref{e} are equivalent by the previous lemma. The equivalence of \ref{c} and \ref{d} is elementary: On the one hand, one easily checks (or looks up in \cite{KT}) that $1+\rho+\rho^2$ has norm $7$; on the other hand, since $(7)$ ramifies, any two (totally positive) elements of norm $7$ generate the same ideal, hence differ only by multiplication by a (totally positive) unit. It is well known \cite{Ti} that in simplest cubic fields, a unit is totally positive if and only if it is a square.

The implication from \ref{e} to \ref{a} is also immediate: An indecomposable element is either a square or not a sum of squares. Hence our task is to prove the converse: If $\alpha$ is decomposable, then $\alpha$ is a sum of squares. In the following, we will repeatedly exploit Corollary \ref{co:minusunitbad}: If a totally positive number is not a sum of squares, then in $\Odya$ it must be equal to $-u^2$ where $u$ is a unit.

Assume that $\alpha = \beta + \gamma$. If both $\beta$ and $\gamma$ are sums of squares, then so is $\alpha$. If neither of them is a sum of squares, then at the dyadic place, $\alpha$ is $-(u_1^2+u_2^2)$ where $u_1,u_2$ are units. However, by Lemma \ref{le:sumoftwosquares} (1) the sum of two unit squares is not a square in $\Odya$, hence $\alpha$ is a sum of squares in $\O$.

Hence we may assume that $\beta = \sum x_i^2$ and $\gamma$ is not a sum of squares. We also may assume that $x_1$ is a unit in $\Odya$; if not, we use the equality $(2x)^2=x^2+x^2+x^2+x^2$. Clearly it suffices to show that $\tilde{\gamma} = x_1^2 + \gamma$ is a sum of squares. And this is easy, since Lemma \ref{le:sumoftwosquares} (1) gives a contradiction if both $\gamma$ and $\tilde{\gamma}$ are minus squares of units in $\Odya$.
\end{proof}

\setcounter{section}{5}
\setcounter{theorem}{1}

The key idea behind the above proof is that while not every indecomposable element is a sum of squares, the sum of two indecomposables in $\O$ can always be rewritten as a sum of squares. With the just proven theorem, one immediately gets Corollary \ref{co:universal}: The quadratic form $x_1^2+x_2^2+x_3^2+x_4^2$ represents all totally positive elements except for those of the form $u^2(1+\rho+\rho^2)$, and these are represented by $(1+\rho+\rho^2)x_5^2$.

\section{Final notes}

The crucial ingredient of our proofs was the fact that the sum of three squares $I_3$ satisfies the local-global principle. The same arguments in fact prove that $\P(\O_K)=4$ for every totally real field $K$ of odd degree where $(2)$ is inert, \emph{provided that $I_3$ integrally represents all numbers which it integrally represents everywhere locally}, but the only straightforward way to prove this integral local-global principle is proving $h(I_3)=1$.

Therefore, the only totally real fields where it is reasonably simple to show that $\P(\O_K)$ is small are the six fields listed in Theorem \ref{th:peters} as the only fields with $h(I_3)=1$. Let us discuss them briefly:

For $\Q$, we have $\P(\Z)=4$ by Lagrange's four square theorem. For $\Q(\sqrt2)$ and $\Q(\sqrt5)$ it is well known \cite{Dz, Ma} that $\P(\O_K)=3$; the proofs are based on the local-global principle. For $\Q(\sqrt{17})$, the situation is different: $\P\bigl(\Z\bigl[\frac{1+\sqrt{17}}{2}\bigr]\bigr)=5$, which is the largest value allowed for quadratic orders by the bound $g_{\Z}(2)$; this is easily proven by checking that $7 + \bigl(\frac{1+\sqrt{17}}{2}\bigr)^2$ is not a sum of four squares. Actually, by \cite[Th.\ 5'']{CP}, there are many different elements of length $5$ in $\Z\bigl[\frac{1+\sqrt{17}}{2}\bigr]$: Any totally positive element of the norm $128\cdot 4^k$ has length $5$.

The only remaining field is $K^{(148)}$, the cubic field with discriminant $148$. There, the same strategy as in this paper might be applied -- first, one proves an analogy of Proposition~\ref{pr:localglobal}, characterising sums of three squares, and then hopefully exploits this result to prove $\P(\O_{K^{(148)}})=4$ or $\P(\O_{K^{(148)}})=5$. It should not be too difficult, but the situation is slightly complicated by the fact that $(2)$ ramifies, so the behaviour at the dyadic place is less simple than in our Lemma \ref{le:dyadicI3}. For example, total positivity is no longer the only local condition for a sum of squares.


\begin{thebibliography}{9999} 



\bibitem[BGV]{BGV} K. J. Becher, D. Grimm and J. Van Geel, \textit{Sums of squares in algebraic function fields over a complete discretely valued field}. Pacific J. Math. 276 (2), 257--276 (2014).

\bibitem[BL]{BL} K. J. Becher and D. B. Leep, \textit{The length and other invariants of a real field}. Math. Z. 269, 235--252 (2011).

\bibitem[BV]{BV} K. J. Becher and J. Van Geel, \textit{Sums of squares in function fields of hyperelliptic curves}, Math. Z. 261, 829--844 (2009).






\bibitem[CP]{CP} H. Cohn, G. Pall, \textit{Sums of four squares in a quadratic ring}, Trans. Amer. Math. Soc. 105, 536–556 (1962).


\bibitem[Dz]{Dz} J. Dzewas, \textit{Quadratsummen in reell-quadratischen Zahlkörpern}, Math. Nachr. 21, 233--284 (1960).

\bibitem[EK]{EK} A. Earnest, A. Khosravani. \textit{Universal positive quaternary quadratic lattices over totally real number fields}, Mathematika 44(2), 342--347 (1997).


\bibitem[Ho]{Ho} D. W. Hoffmann, \textit{Pythagoras numbers of fields},  J. Amer. Math. Soc. 12 (3), 839--848 (1999).

\bibitem[Hu]{Hu} Y. Hu, \textit{The Pythagoras number and the $u$-invariant of Laurent series fields in several variables}, J. Algebra 426, 243--258 (2015).






\bibitem[KO]{KO} M.-H. Kim, B.-K. Oh, \textit{Bounds for quadratic Waring’s problem}, Acta Arith. 104, 155--164 (2002).



\bibitem[Kr]{Kr} J. Krásenský, \textit{Pythagoras number of non-maximal orders in biquadratic fields}, in preparation.


\bibitem[KT]{KT} V. Kala, M. Tinková, \textit{Universal quadratic forms, small norms and traces in families of number fields}, submitted. Available at \href{http://arxiv.org/abs/2005.12312}{arXiv:2005.12312}.

\bibitem[KY]{KY} V. Kala and P. Yatsyna, \textit{Lifting problem for universal quadratic forms}, Adv. Math. 377, 24 pp. (2021).


\bibitem[KRS]{KRS} J. Kr\'asensk\'y, M. Ra\v{s}ka, E. Sgallov\'a, \textit{Pythagoras numbers of orders in biquadratic fields}, submitted. Available at \href{https://arxiv.org/abs/2105.08860}{arXiv:2105.08860}.


\bibitem[La]{La} T. Y. Lam, \textit{Introduction to quadratic forms over fields}, Graduate Studies in Mathematics 65, American Mathematical Society (2005).

\bibitem[Ma]{Ma} H. Maa{\ss}, \textit{{\"U}ber die Darstellung total positiver Zahlen des K{\"o}rpers  R($\sqrt5$) als Summe von drei Quadraten}, Abh. Math. Sem. Univ. Hamburg 14, 185--191 (1941).

\bibitem[Mi]{Mi} J. S. Milne, \textit{Algebraic Number Theory} (v3.08), 2020, Available at \url{www.jmilne.org/math/}.




\bibitem[O'M]{OMeara} O. T. O’Meara, \textit{Introduction to Quadratic Forms}, Springer-Verlag (1973).


\bibitem[Pe]{Pe} M. Peters, \textit{Summen von Quadraten in Zahlringen}, J. Reine Angew. Math. 268/269, 318--323 (1974).

\bibitem[Pe2]{Pe2} M. Peters, \textit{Einklassige Geschlechter von Einheitsformen in totalreellen algebraischen Zahlk\"{o}rpern}. Math. Ann. 226, 117--120 (1977).

\bibitem[Pf]{Pf} A. Pfister, \textit{Quadratic forms with applications to algebraic geometry and topology}, London Math. Soc. Lect. Notes 217, Cambridge University Press (1995).





\bibitem[Sch]{Sch} R. Scharlau, \textit{On the Pythagoras number of orders in totally real number fields}, J. Reine Angew. Math. 316, 208--210 (1980).

\bibitem[Sch2]{Sch2} R. Scharlau, \textit{Zur Darstellbarkeit von totalreellen ganzen algebraischen Zahlen durch Summen von Quadraten}, dissertation at Universit{\"a}t Bielefeld (1979).

\bibitem[Sh]{Sh} D. Shanks, \textit{The simplest cubic number fields}, Math. Comp. 28, 1137--1152 (1974).


\bibitem[Ti]{Ti} M. Tinkov\'a, \textit{On the Pythagoras number of the simplest cubic fields}, submitted. Available at	\href{https://arxiv.org/abs/2101.11384}{arXiv:2101.11384}.

\bibitem[TVY]{TVY} S. V. Tikhonov, J. Van Geel and V. I. Yanchevskii, \textit{Pythagoras numbers of function fields of hyperelliptic curves with good reduction}, Manuscripta Math. 119, 305--322 (2006).





\end{thebibliography}
\end{document}